\documentclass[11pt,leqno]{article}
\usepackage{amsmath,amssymb,amsthm,latexsym}
\usepackage{indentfirst}
\usepackage{amsmath}

\newtheorem{theorem}{\bf \large Theorem}[section]
\newtheorem{PROPOSITION}{\bf \large Proposition}[section]
\newtheorem{corollary}{\bf \large Corollary}[section]

\newtheorem{ex}{\bf \large Example}[section]

\title{\textbf Para-Blaschke isoparametric spacelike hypersurfaces in  Lorentzian space forms}
\author {{ Xiu Ji,~~~Tongzhu Li,~~~Huafei Sun} \\
\small{Department of Mathematics, Beijing Institute of
Technology,} \\
\small{Beijing, 100081, China.}\\
\small{Beijing Key Laboratory on MCAACI,}
\small{Beijing,100081,China.} \\
\small{ E-mail: jixiu1106@163.com,~~litz@bit.edu.cn,~~huafeisun@bit.edu.cn.}}
\date{}

\begin{document}
\maketitle
\begin{abstract}
Let $M^n$ be an $n$-dimensional umbilic-free hypersurface in the $(n+1)$-dimensional Lorentzian
space form $M^{n+1}_1(c)$. Three  basic invariants of $M^n$ under the
conformal transformation group of $M^{n+1}_1(c)$ are a $1$-form $C$, called conformal $1$-form, a symmetric
$(0,2)$ tensor $B$, called conformal second fundamental form, and a
symmetric $(0,2)$ tensor $A$, called Blaschke tensor. The so-called para-Blaschke tensor $D^{\lambda}=A+\lambda B$, the linear combination of $A$ and $B$,
is still a symmetric $(0,2)$ tensor. A spacelike hypersurface is called a para-Blaschke isoparametric spacelike hypersurface, if the conform $1$-form vanishes and
the eigenvalues of the
para-Blaschke tensor are constant.
In this paper, we classify the para-Blaschke isoparametric spacelike hypersurfaces under the conformal group of $M^{n+1}_1(c)$.
\end{abstract}

\medskip\noindent
{\bf 2000 Mathematics Subject Classification:} 53A30, 53B25.
\par\noindent {\bf Key words:} Blaschke tnesor,
para-Blaschke tensor, para-Blaschke isoparametric hypersurface, conformal isoparametric hypersurface.

\vskip 1 cm
\section{Introduction}
\par\medskip
Recently the M\"{o}bius geometry of submanifolds in Riemannian space forms has been studied extensively and a lot of
interesting results have been obtained. Especially,  Many special hypersurfaces were classified under M\"{o}bius transformation group (for example, \cite{cheng},\cite{guo},\cite{guo1},\cite{hu},\cite{hu1},\cite{lih},\cite{lih3},\cite{lix2},\cite{lix3}).
As its parallel generalization, the conformal geometry of submanifolds in Lorentzian space forms is another important branch of
conformal geometry, but there are less results in  Lorentzian space forms than in Riemannian space forms.
In this paper, we study the para-Blaschke isoparametric spacelike hypersurfaces in Lorentzian space forms.

Let $ \mathbb{R}^{n+2}_s$ be the real vector space $ \mathbb{R}^{n+2}$ with the
Lorentzian  product $\langle,\rangle_s$ given by
$$\langle X,Y\rangle_s=-\sum_{i=1}^sx_iy_i+\sum_{j=s+1}^{n+2}x_jy_j.$$
 For any $a>0$, the standard sphere $\mathbb{S}^{n+1}(a)$, the hyperbolic space $\mathbb{H}^{n+1}(-a)$,
the de sitter space $\mathbb{S}^{n+1}_1(a)$ and the anti-de sitter space $\mathbb{H}^{n+1}_1(-a)$ are defined by
\begin{equation*}
\begin{split}
\mathbb{S}^{n+1}(a)=\{x\in \mathbb{R}^{n+2}|x\cdot x=a^2\},~~\mathbb{H}^{n+1}(-a)=\{x\in \mathbb{R}^{n+2}_1|\langle x,x\rangle_1=-a^2\},\\
\mathbb{S}^{n+1}_1(a)=\{x\in \mathbb{R}^{n+2}_1|\langle x,x\rangle_1=a^2\},~~ \mathbb{H}^{n+1}_1(-a)=\{x\in
\mathbb{R}^{n+2}_2|\langle x,x\rangle_2=-a^2\}.
\end{split}
\end{equation*}
Let $M^{n+1}_1(c)$ be a Lorentzian space form. When $c=0$,
$M^{n+1}_1(c)=\mathbb{R}^{n+1}_1$. When $c=1$, $M^{n+1}_1(c)=\mathbb{S}^{n+1}_1(1)$.
When $c=-1$, $M^{n+1}_1(c)=\mathbb{H}^{n+1}_1(-1)$.

For Lorentzian space forms $M^{n+1}_1(c)$, there exists a united conformal
compactification $\mathbb Q^{n+1}_1$, which is the projectivized
light cone in $\mathbb RP^{n+2}$ induced from $\mathbb{R}^{n+3}_2$. Using the conformal compactification $\mathbb Q^{n+1}_1$,
we study the conformal geometry of spacelike hypersurfaces in $M_1^{n+1}(c)$. We define the conformal metric $g$ and  the conformal second
fundamental form $B$ on an umbilic-free spacelike hypersurface, which determine the spacelike hypersurface up
to a conformal transformation of $M^{n+1}_1(c)$. Another two conformal invariants are the conformal $1$-form $C$ and the Blaschke tensor $A$ (see Sect.2).

Since  $A$ and $B$ are symmetric $(0,2)$-tensor, their eigenvalues are real. We define two kind of special spacelike hypersurfaces: the conformal isoparametric
spacelike  hypersurfaces and the Blaschke isoparametric spacelike hypersurfaces. A spacelike hypersurface is called a conformal isoparametric spacelike hypersurface, if it satisfies two conditions: (1) $C = 0,$ (2) all the eigenvalues of $B$ are constant. Similarly, we define the Blaschke isoparametric
spacelike hypersurface by another symmetric tensor, the Blaschke tensor $A$.
The para-Blaschke tensor defined by $D^{\lambda}:=A+\lambda B$ for some constant $\lambda$. Clearly the para-Blaschke tensor is still a symmetric $(0,2)$ tensor, thus its eigenvalues are real. Using the para-Blaschke tensor, we can define similarly the para-Blaschke isoparametric spacelike hypersurface.

Recently, some interesting resultes on the spacelike hypersurfaces with some special conformal invariants are obtained.
C. X.  Nie et al. classified the spacelike  hypersurfaces with parallel conformal second fundamental form in \cite{nie2}, and classified the Blaschke
 isoparametric  spacelike hypersurfaces with two distinct principal curvatures in \cite{nie1}. X. X. Li et al. classified the spacelike  hypersurfaces with parallel Blaschke
tensor in \cite{lix1} and the spacelike hypersurfaces with with parallel para-Blaschke
tensor in \cite{lix4}.  T.Z. Li and C.X. Nie classified completely the conformal isoparametric spacelike hypersurfaces in \cite{lin1}. Clearly if the para-blaschke tensor of a spacelike hypersurface is parallel, then the spacelike hypersurface is para-Blaschke isoparametric.
In this paper, we prove that a para-Blaschke isoparametric spacelike  hypersurface is a conformal
isoparametric spacelike  hypersurface provided that the para-Blaschke tensor has more than two distinct
eigenvalues. Simultaneously, we classify completely the para-Blaschke isoparametric spacelike  hypersurfaces.
Our main theorems are as follows.
\begin{theorem}\label{theorem1}
Let $x:M^n\to M^{n+1}_1(c), n\geq 2,$ be an umbilic-free spacelike hypersurface in an $(n+1)$-dimensional Lorentzian
space form $M^{n+1}_1(c)$. We assume that the conformal $1$-form of $x$ vanishes. Then we have\\
(1),If the spacelike hypersurface is a conformal isoparametric spacelike hypersurface, then the spacelike hypersurface is also a para-Blaschke isoparametric spacelike  hypersurface.\\
(2),If the spacelike hypersurface is  a para-Blaschke isoparametric spacelike  hypersurface and the number
of the distinct eigenvalues of the para-Blaschke tensor $D^{\lambda}$ is more than two,
then the spacelike hypersurface is also a conformal isoparametric spacelike hypersurface.
\end{theorem}
\begin{theorem}\label{theorem2}
Let $x:M^n\to M^{n+1}_1(c), n\geq 2,$ be an umbilic-free spacelike hypersurface in an $(n+1)$-dimensional Lorentzian
space form $M^{n+1}_1(c)$. If the hypersurface is para-Blaschke isoparametric, then $x$ is locally conformal equivalent to one of the following hypersurfaces:\\
(1), the spacelike hypersurfaces with constant mean curvature
and constant scalar curvature in $M^{n+1}_1(c)$;\\
(2), $\mathbb{S}^k(\sqrt{a^2+1})\times\mathbb{H}^{n-k}(-a)\subset\mathbb{S}^{n+1}_1(1), ~~a>0, ~1\leq k\leq n-1;$\\
(3), $\mathbb{H}^k(-a)\times\mathbb{H}^{n-k}(-\sqrt{1-a^2})\subset\mathbb{H}^{n+1}_1(-1),~~0<a<1,~1\leq k\leq n-1;$\\
(4), $\mathbb{H}^k(-a)\times\mathbb{R}^{n-k}
\subset\mathbb{R}^{n+1}_1,~~a>0,~0\leq k\leq n-1;$\\
(5),$
 x:\mathbb H^{q}(-\sqrt{a^2-1})\times\mathbb S^{p}(a)\times
 \mathbb R^+\times\mathbb R^{n-p-q-1}\rightarrow
 \mathbb R^{n+1}_1 ,$ defined by $$
  x(u',u'',t,u''')=(tu',tu'',u'''),$$ where $u'\in \mathbb H^{q}(-\sqrt{a^2-1}), u''\in\mathbb S^{p}(a),u'''\in\mathbb R^{n-p-q-1},~~a>1;$\\
(6) the spacelike hypersurfaces defined by Example \ref{ex5} (see Sect.3);\\
(7) the spacelike hypersurfaces defined by Example \ref{ex6} (see Sect.3).
\end{theorem}
When $\lambda=0$, $D^{\lambda}=A$. Theorem \ref{theorem2} implies
that the conformal isoparametric spacelike hypersurfaces and the
Blaschke isoparametric spacelike hypersurfaces are almost equivalent.
Therefore from the results in \cite{lin1}, we have the following results.
\begin{corollary}\label{corol}
Let $x:M^n\to M^{n+1}_1(c), n\geq 2,$ be an umbilic-free spacelike hypersurface in the $(n+1)$-dimensional Lorentzian
space form $M^{n+1}_1(c)$ with $r$ distinct eigenvalues of the Blaschke tensor. If the hypersurface is  Blaschke isoparametric and $r\geq 3$, then $r=3$ and $x$ is locally conformal equivalent to the following spacelike hypersurface:
$$ x:\mathbb H^{q}(-\sqrt{a^2-1})\times\mathbb S^{p}(a)\times
 \mathbb R^+\times\mathbb R^{n-p-q-1}\rightarrow
 \mathbb R^{n+1}_1,$$ defined by $
  x(u',u'',t,u''')=(tu',tu'',u'''),$ where $u'\in \mathbb H^{q}(-\sqrt{a^2-1}), u''\in\mathbb S^{p}(a),u'''\in\mathbb R^{n-p-q-1},~~a>1.$
\end{corollary}
This paper is organized as follows. In section 2, we study the conformal geometry of spacelike hypersurfaces in $M_1^{n+1}(c)$. In section 3, we
give  some examples of  special  spacelike hypersurfaces. In
section 4  we give the proof of our main theorems.

\par\noindent
\section{Conformal geometry of spacelike Hypersurfaces}
\par\medskip
In this section, following Wang's idea in paper \cite{w}, we define some conformal invariants on a spacelike  hypersurface
and give a congruent theorem of the spacelike hypersurfaces under the conformal
group of $M^{n+1}_1(c)$.

We denote by $C^{n+2}$ the cone in $\mathbb R^{n+3}_2$ and by
$\mathbb Q^{n+1}_1$ the conformal
compactification space in $\mathbb RP^{n+2}$,
$$C^{n+2}=\{X\in\mathbb R^{n+3}_2|\langle X,X\rangle_2=0,X\neq0\},$$
$$\mathbb Q^{n+1}_1=\{[X]\in\mathbb R P^{n+2}|\langle X,X\rangle_2=0\}.$$
Let $O(n+3,2)$ be the Lorentzian group of $\mathbb{R}^{n+3}_2$ keeping the
Lorentzian product $\langle X,Y\rangle_2$ invariant. Then $O(n+3,2)$ is a
transformation group on $\mathbb Q^{n+1}_1$ defined by
\begin{equation*}
T([X])=[XT], ~~~X\in C^{n+2}, ~~~T\in O(n+3,2).
\end{equation*}
Topologically $\mathbb Q ^{n+1}_1$ is identified with the compact
space $S^n\times S^1/S^0$, which is endowed by a standard Lorentzian
metric $h=g_{S^n}\oplus(-g_{S^1})$, where $g_{S^k}$ denotes the standard metric of the $k$-dimensional sphere $S^k$.  Then $\mathbb Q^{n+1}_1$ has
conformal metric $$[h]=\{e^\tau h|\tau\in C^ \infty(\mathbb
Q^{n+1}_1)\}$$ and $[O(n+3,2)]$ is the conformal transformation group of
$\mathbb Q^{n+1}_1$(see\cite{cahne,o}).

Denoting $P=\{[X]\in \mathbb Q ^{n+1}_1|x_1=x_{n+2}\},~~P_-=\{[X]\in \mathbb Q ^{n+1}_1|x_{n+2}=0\},~~P_+=\{[X]\in \mathbb Q ^{n+1}_1|x_1=0\}$,
we can define the following conformal diffeomorphisms,
\begin{equation*}
\begin{array}{l}
 \sigma_0 :
\mathbb{R}^{n+1}_1\rightarrow {\mathbb Q}^{n+1}_1\backslash P,~~~~\quad u\mapsto[( \frac{< u,
u>_1+1}{2},u, \frac{<u, u>_1-1}{2})],
   \\
\sigma_1: \mathbb{S}^{n+1}_1(1)\rightarrow  {\mathbb Q}^{n+1}_1\backslash P_+ ,~~\quad
 u\mapsto [(1,u)],  \\
\sigma_{-1}:  \mathbb{H}^{n+1}_1(-1)\rightarrow {\mathbb Q}^{n+1}_1\backslash P_-,~\quad
  u\mapsto [(u,1)]. \\
   \end{array}
\end{equation*}
We may regard $\mathbb {Q}^{n+1}_1$ as the common compactification of $\mathbb{R}^{n+1}_1, \mathbb{S}^{n+1}_1(1), \mathbb{H}^{n+1}_1(-1)$.

Let $x:M^n\rightarrow M^{n+1}_1(c)$ be a spacelike hypersurface. Using $\sigma_c$, we obtain the hypersurface
in $\mathbb{Q}^{n+1}_1$, $\sigma_c\circ x:M^n\rightarrow \mathbb {Q}^{n+1}_1$.  From \cite{cahne}, we have the following theorem.
\begin{theorem}
Two hypersurfaces $x,\bar{x}:M^n\rightarrow M^{n+1}_1(c)$ are
conformally equivalent if and only if there exists $T\in O(n+3,2)$ such
that $\sigma_c\circ x=T(\sigma_c\circ \bar{x}):M^n\rightarrow \mathbb
Q^{n+1}_1$.
\end{theorem}
Since $x:M^n\rightarrow M^{n+1}_1(c)$ is a spacelike hypersurface,  $(\sigma_c\circ x)_*(TM^n)$ is a positive
definite subbundle of $T{\mathbb Q}^{n+1}_1$. For any local lift $Z$
of the standard projection $\pi: C^{n+2}\rightarrow
\mathbb{Q}^{n+1}_1$, we get a local lift $y=Z\circ\sigma_c\circ
x:U\rightarrow C^{n+1}$  of $\sigma_c\circ x:
M\rightarrow{\mathbb Q}^{n+1}_1$ in an open subset $U$ of
$M^n$. Thus $\langle\text dy, \text dy\rangle_2=\rho^2 \langle dx, dx\rangle_s$ is a local metric, where $\rho\in C^{\infty}(U)$. We denote by $\Delta$ and $\kappa$ the Laplacian
operator and the normalized scalar curvature with respect to the local positive
definite
metric $\langle\text dy, \text dy\rangle$, respectively.  Similar to Wang's proof of Theorem 1.2 in \cite{w},  we can get
the following theorem.
\begin{theorem}\label{t22}
Let $x:M^n\rightarrow M^{n+1}_1(c)$ be a spacelike hypersurface, then the 2-form $g=-(\langle\Delta y, \Delta
y\rangle_2-n^2\kappa)\langle\text dy, \text dy\rangle_2 $ is a globally
defined conformal invariant.
Moreover, $g$ is positive definite at any non-umbilical point of
$M^n$.
\end{theorem}
We call $g$
the conformal metric of the spacelike  hypersurface $M^n$. There exists a unique lift
$$Y:M\rightarrow C^{n+2}$$ such that $g=\langle\text dY, \text dY\rangle_2$.  We call $Y$ the conformal position
vector of the spacelike  hypersurface $M^n$.
Theorem \ref{t22} implies that
\begin{theorem}
Two spacelike hypersurfaces $x,\bar{x}:M^n\rightarrow M^{n+1}_1(c)$
are conformally equivalent if and only if there exists $T\in O(n+3,2)$
such that $\bar{Y}=YT$, where $Y,\tilde Y$ are the conformal
position vector of $x,\tilde{x}$, respectively.
\end{theorem}

Let $\{E_1, \cdots , E_n\}$ be a local orthonormal basis of $M^n$
with respect to $g$ with dual basis $\{\omega_1, \cdots ,
\omega_n\}$. Denote $Y_i=E_i(Y)$ and define
\begin{equation*}
N=-\frac{1}{n}\Delta Y-\frac{1}{2n^2}\langle\Delta Y, \Delta
Y\rangle_2 Y,
\end{equation*}
where $\Delta$ is the Laplace operator of $g$, then we have
\begin{equation*}
\langle N, Y\rangle_2=1,~ \langle N, N\rangle_2=0,~\langle N,
Y_k\rangle_2=0,~ \langle Y_i, Y_j\rangle_2=\delta_{ij},\quad1\leq i,j,k\leq n.
\end{equation*}
We may decompose $\mathbb{R}^{n+3}_2$ such that
$$\mathbb{R}^{n+3}_2=\text{span}\{Y, N\}\oplus \text{span}
\{Y_1, \cdots , Y_n\}\oplus\mathbb V ,$$ where $\mathbb
V\bot\text{span}\{Y, N, Y_1, \cdots , Y_n\}$.  We call $\mathbb V$
the conformal normal bundle of $x$,
which is linear bundle. Let $\xi$ be a local section of $\mathbb{V}$
and $<\xi,\xi>_2=-1$, then $\{Y, N, Y_1, \cdots , Y_n, \xi\}$ forms a
moving frame in $\mathbb{R}^{n+3}_2$ along $M^n$. We write the structure
equations as follows,
\begin{equation}\label{struct}
\begin{split}
&\mathrm{d}Y=\sum_i\omega_iY_i,\\
&\mathrm{d}N=\sum_{ij}A_{ij}\omega_jY_i+\sum_iC_i\omega_i\xi,\\
&\mathrm{d}Y_i=-\sum_{j}A_{ij}\omega_jY-\omega_iN+\sum_j\omega_{ij}Y_j+\sum_{j}B_{ij}\omega_j\xi,\\
&\mathrm{d}\xi=\sum_iC_i\omega_iY+\sum_{ij}B_{ij}\omega_jY_i,
\end{split}
\end{equation}
 where $\omega_{ij} (=-\omega_{ij})$ are the connection 1-forms on $M^n$ with respect to $\{\omega_1, \cdots ,
\omega_n\}$.  It is clear that
 $A=\sum_{ij}A_{ij}\omega_j\otimes\omega_i,~
B=\sum_{ij}B_{ij}\omega_j\otimes\omega_i,~ C=\sum_iC_i\omega_i$ are globally defined conformal
invariants. We call $A,~B$ and $C$ the Blaschke tensor, the conformal second fundamental form
 and the conformal $1$-form, respectively. The covariant derivatives
of these tensors with respect to  $\omega_{ij}$ are defined by:
$$\sum_jC_{i, j}\omega_j=dC_i+\sum_kC_k\omega_{kj},$$
$$\sum_kA_{ij, k}\omega_k=dA_{ij}+\sum_kA_{ik}\omega_{kj}
+\sum_kA_{kj}\omega_{ki},$$
$$\sum_kB_{ij, k}\omega_k=dB_{ij}+\sum_kB_{ik}\omega_{kj}
+\sum_kB_{kj}\omega_{ki}.$$
By
exterior differentiation of structure equations (\ref{struct}),
we can get the integrable conditions of the structure equations
$$A_{ij}=A_{ji},~~~B_{ij}=B_{ji},$$
\begin{equation}\label{stru1}
A_{ij, k}-A_{ik, j}=B_{ij}C_k-B_{ik}C_j,
\end{equation}
\begin{equation}\label{stru2}
B_{ij,k}-B_{ik,
j}=\delta_{ij}C_k- \delta_{ik}C_j,
\end{equation}
\begin{equation}\label{stru3}
C_{i, j}-C_{j, i}=\sum_k(B_{ik}A_{kj}-B_{jk}A_{ki}),
\end{equation}
\begin{equation}\label{stru4}
R_{ijkl}=B_{il}B_{jk}-B_{ik}B_{jl}+A_{ik}\delta_{jl}+A_{jl}\delta_{ik}-A_{il}\delta_{jk}-A_{jk}\delta_{il}.
\end{equation}
Furthermore, we have
\begin{equation}\label{cond1}
\begin{split}
&\text{tr}(A)=\frac{1}{2n}( n^2\kappa-1),\quad
R_{ij}=\text{tr}(A)\delta_{ij}+(n-2)A_{ij}+\sum_kB_{ik}
B_{kj},\\
&(1-n)C_i=\sum_jB_{ij,j},\quad \sum_{ij}B_{ij}^2=\frac{n-1}{n}, \quad
\sum_i B_{ii}=0,
\end{split}
\end{equation}
where $\kappa$ is the normalized scalar curvature of $g$. From
(\ref{cond1}), we see that when $n\geq3$, all coefficients in the
structure equations are determined by the conformal
metric $g$ and the conformal second fundamental form $B$, thus we get the following conformal congruent
theorem.
\begin{theorem}
Two spacelike hypersurfaces
$x, \bar{x}: M^n\rightarrow M^{n+1}_1(c) (n\geq3)$ are conformally
equivalent if and only if there exists a diffeomorphism $\varphi:
M^n\rightarrow M^n$ which preserves the conformal metric  and the
conformal second fundamental form.
\end{theorem}
Next we give the relations between the conformal invariants and the
isometric invariants of a spacelike hypersurface in $M^{n+1}_1(c)$.

First we consider the spacelike hypersurface  $x:M^n\rightarrow \mathbb{R}^{n+1}_1$ in  $\mathbb{R}^{n+1}_1$. Let $\{e_1, \cdots , e_n\}$ be an orthonormal
local basis with respect to the induced metric $I=<dx,dx>_1$ with dual basis
$\{\theta_1, \cdots , \theta_n\}$. Let $e_{n+1}$ be a normal vector
field  of $x$ , $<e_{n+1},e_{n+1}>_1=-1$. Let $II=\sum_{ij}h_{ij}\theta_i\otimes\theta_j$ denote the second fundamental form,  the mean curvature
$ H=\frac{1}{n}\sum_{i}
h_{ii}$. Denote by $\Delta_{M}$ the Laplacian operator and $\kappa_{M}$ the
normalized scalar curvature for $I$.  By structure equation of $x:M^n\rightarrow \mathbb{R}^{n+1}_1$ we get that
\begin{equation}\label{hh}
\Delta_{M}x=nHe_{n+1}.
\end{equation}
There is a local lift of $x$
$$y:M^n\rightarrow C^{n+2},\quad y=( \frac{<x, x>_1+1}{2},x, \frac{<x, x>_1-1}{2})
.$$
It follows from (\ref{hh}) that
$$\langle\Delta y, \Delta y\rangle_2-n^2\kappa_M=\frac{n}{n-1}(-|II|^2+n| H|^2)=-e^{2\tau}.$$
Therefore the conformal metric $g$, conformal position vector of $x$ and $\xi$
have the following expression,
\begin{equation}\label{coff0}
\begin{split}
&g=\frac{n}{n-1}(|II|^2-n|H|^2)<\text dx,\text dx>_1:=e^{2\tau}I,~~~
Y=e^\tau y,\\
&\xi=-H y+(<x,e_{n+1}>_1, e_{n+1},<x,e_{n+1}>_1).
\end{split}
\end{equation}
By a direct calculation we get the following expression of the
conformal invariants,
\begin{equation}\label{coff}
\begin{split}
&A_{ij}=e^{-2\tau}[\tau_i\tau_j- h_{ij}H
-\tau_{i,j}+\frac{1}{2}(-|\nabla\tau|^2+| H|^2) \delta_{ij}],\\
&B_{ij}=e^{-\tau}(h_{ij}-H \delta_{ij}),~~~
C_i=e^{-2\tau}(H\tau_i-H_i
-\sum_{j}h_{ij}\tau_j),
\end{split}
\end{equation}
where $\tau_i=e_i(\tau)$ and $|\nabla \tau|^2=\sum_i\tau_i^2$, and $\tau_{i,j}$ is the Hessian of $\tau$ for $I$ and $H_
i=e_i(H)$.

For a spacelike hypersurface $x:M^n\rightarrow
\mathbb{S}^{n+1}_1(1)$,
the conformal metric $g$, conformal position vector of $x$ and $\xi$
have the following expression,
\begin{equation}\label{coff01}
\begin{split}
&g=\frac{n}{n-1}(|II|^2-n|H|^2)<\text dx,\text dx>_1:=e^{2\tau}I,\\
&Y=e^\tau(1,x)=e^\tau y,~~~~
\xi=-H y+(0,e_{n+1}).
\end{split}
\end{equation}

For a spacelike hypersurface $x:M^n\rightarrow \mathbb{H}^{n+1}_1(-1)$,
the conformal metric $g$, conformal position vector of $x$ and $\xi$
have the following expression,
\begin{equation}\label{coff02}
\begin{split}
&g=\frac{n}{n-1}(|II|^2-n|H|^2)<\text dx,\text dx>_2:=e^{2\tau}I,\\
&Y=e^\tau(x,1)=e^\tau y,~~~
\xi=-H y+(e_{n+1},0).
\end{split}
\end{equation}
Using the similar calculation from (\ref{coff01}) and (\ref{coff02}), we have the
following united expression of the conformal invariants,
\begin{equation}\label{coff1}
\begin{split}
&A_{ij}=e^{-2\tau}[\tau_i\tau_j-\tau_{i,j}- h_{ij}H
+\frac{1}{2}(-|\nabla\tau|^2+| H|^2+c) \delta_{ij}],\\
&B_{ij}=e^{-\tau}(h_{ij}-H \delta_{ij}),~~~
C_i=e^{-2\tau}(H\tau_i-H_i
-\sum_{j}h_{ij}\tau_j),
\end{split}
\end{equation}
where $c=1$ for $x:M^n\rightarrow S^{n+1}_1(1)$, and
$c=-1$ for $x:M^n\rightarrow H^{n+1}_1(-1)$.

\par\noindent
\section{Typical examples}
In this section, we present some examples of the
spacelike  hypersurfaces in $M^{n+1}_1(c)$ with constant eigenvalues of para-Blaschke tensor.
\begin{ex}\label{ex1}
For constant $a>0$, let $x_1:\mathbb{H}^k(-1)\to \mathbb{R}^{k+1}_1$ be the standard embedding
and $y:\mathbb{R}^{n-k}\to \mathbb{R}^{n-k}$ identity. We define the spacelike hypersurface  $$x=(x_1,y):\mathbb{H}^k(-a)\times\mathbb{R}^{n-k}
\to\mathbb{R}^{n+1}_1,~~1\leq k\leq n-1.$$
\end{ex}
Let $\xi=(\frac{1}{a}x_1,\overrightarrow{0})$ be the normal vector field of $x$. Thus
$$I=<dx,dx>_1=I_{\mathbb{H}^k(-a)}+I_{\mathbb{R}^{n-k}},~~~II=-<dx,d\xi>_1=\frac{-1}{a}I_{\mathbb{H}^k(-a)},$$
where $I_{\mathbb{H}^k(-a)}$ denotes the standard metric on $\mathbb{H}^k(-a)$ and $I_{\mathbb{R}^{n-k}}$ the standard metric on $\mathbb{R}^{n-k}$.

 Let $\{e_1,\cdots,e_k\}$ be
a local fields of orthonormal basis on $\mathbb{H}^k(-a)$ and $\{e_{k+1},\cdots,e_n\}$ a local fields of orthonormal basis on $\mathbb{R}^{n-k}$,
then $\{e_1,\cdots,e_n\}$ is a local fields of orthonormal basis on $\mathbb{H}^k(-a)\times\mathbb{R}^{n-k}$. Thus, under the local fields of orthonormal basis $\{e_1,\cdots,e_n\}$,
$$\big(h_{ij}\big)=diag(\frac{-1}{a},\cdots,\frac{-1}{a},0,\cdots,0).$$
Under the local fields of orthonormal basis, from (\ref{coff}),  we have
$$(B_{ij})=diag(\underbrace{b_1,\cdots,b_1}_k,\underbrace{b_2,\cdots,b_2}_{n-k}),~~(A_{ij})=diag(\underbrace{a_1,\cdots,a_1}_k,\underbrace{a_2,\cdots,a_2}_{n-k}),$$
where $$b_1=\frac{1}{n}\sqrt{\frac{(n-1)(n-k)}{k}}, ~b_2=\frac{-1}{n}\sqrt{\frac{(n-1)k}{n-k}}, ~a_1=\frac{(n-1)(k-2n)}{2n^2(n-k)}, ~a_2=\frac{(n-1)k}{2n^2(n-k)}.$$
Thus $(D^{\lambda}_{ij})=diag(\underbrace{d_1,\cdots,d_1}_k,\underbrace{d_2,\cdots,d_2}_{n-k})$ and $d_1=a_1+\lambda b_1, ~d_2=a_2+\lambda b_2$.

\begin{ex}\label{ex2}
Let $x_1:\mathbb{S}^k(1)\to \mathbb{R}^{k+1}$ and $x_2:\mathbb{H}^{n-k}(-1)\to \mathbb{R}^{n-k+1}_1$ be two standard embedings.
For constant $a>0$, we define the spacelike hypersurface
$$x=(\sqrt{1+a^2}x_1,ax_2):\mathbb{S}^k(\sqrt{1+a^2})\times\mathbb{H}^{n-k}(-a)\to\mathbb{S}^{n+1}_1(1)\subset \mathbb{R}^{n+2}_1,~~1\leq k\leq n-1.$$
\end{ex}
Let $\xi=(ax_1,\sqrt{1+a^2}x_2)$ be  the normal vector field of $x$. Thus
$$I=<dx,dx>_1=(1+a^2)I_{S^k(1)}+a^2I_{H^{n-k}(-1)},$$
$$II=-<dx,d\xi>_1=-a\sqrt{1+a^2}(I_{S^k(1)}+I_{H^{n-k}(-1)}).$$
Let $\{e_1,\cdots,e_k\}$ be
a local fields of orthonormal basis on $\mathbb{S}^k(\sqrt{1+a^2})$ and $\{e_{k+1},\cdots,e_n\}$ a local fields of orthonormal basis on $\mathbb{H}^{n-k}(-a)$,
then $\{e_1,\cdots,e_n\}$ is a local fields of orthonormal basis on $\mathbb{S}^k(\sqrt{1+a^2})\times\mathbb{H}^{n-k}(-a)$. Thus, under the local fields of orthonormal tangent frame $\{e_1,\cdots,e_n\}$,
$$\big(h_{ij}\big)=diag(\frac{-a}{\sqrt{1+a^2}},\cdots,\frac{-a}{\sqrt{1+a^2}},\frac{-\sqrt{1+a^2}}{a},\cdots,\frac{-\sqrt{1+a^2}}{a}).$$
Under the local fields of orthonormal basis, from (\ref{coff02}),  we have
$$(B_{ij})=diag(\underbrace{b_1,\cdots,b_1}_k,\underbrace{b_2,\cdots,b_2}_{n-k}),~~(A_{ij})=diag(\underbrace{a_1,\cdots,a_1}_k,\underbrace{a_2,\cdots,a_2}_{n-k}),$$
where
\begin{equation*}
\begin{split}
&b_1=\frac{1}{n}\sqrt{\frac{(n-1)(n-k)}{k}}, ~~b_2=\frac{-1}{n}\sqrt{\frac{(n-1)k}{n-k}},\\
&a_1=\frac{n-1}{k(n-k)}\frac{(n-k)^2+n^2a^2}{2n^2}, ~~~a_2=\frac{n-1}{k(n-k)}\frac{k^2-n^2a^2-n^2}{2n^2}.
\end{split}
\end{equation*}
Thus $(D^{\lambda}_{ij})=diag(\underbrace{d_1,\cdots,d_1}_k,\underbrace{d_2,\cdots,d_2}_{n-k})$ and $d_1=a_1+\lambda b_1, ~d_2=a_2+\lambda b_2$.
\begin{ex}\label{ex3}
Let $x_1:\mathbb{H}^k(-1)\to \mathbb{R}^{k+1}_1$ and $x_2:\mathbb{H}^{n-k}(-1)\to \mathbb{R}^{n-k+1}_1$ be two standard embedings. For constant $a$ satisfying $0<a<1$, We we define the spacelike hypersurface
$$x=(\sqrt{1-a^2}x_1,ax_2):\mathbb{H}^k(-a)\times\mathbb{H}^{n-k}(-\sqrt{1-a^2})\to\mathbb{H}^{n+1}_1(-1)\subset \mathbb{R}^{n+2}_2,~~1\leq k\leq n-1.$$
\end{ex}
Let $\xi=(-ax_1,\sqrt{1-a^2}x_2)$ be the normal vector field of $x$. Thus
$$I=<dx,dx>_1=(1-a^2)I_{H^k(-1)}+a^2I_{H^{n-k}(-1)},$$
$$II=-<dx,d\xi>_1=a\sqrt{1-a^2}(I_{H^k(-1)}-I_{H^{n-k}(-1)}).$$
Let $\{e_1,\cdots,e_k\}$ be
a local fields of orthonormal basis on $\mathbb{H}^k(-a)$ and $\{e_{k+1},\cdots,e_n\}$ a local fields of orthonormal basis on $\mathbb{H}^{n-k}(-\sqrt{1-a^2})$,
then $\{e_1,\cdots,e_n\}$ is a local fields of orthonormal basis on $\mathbb{H}^k(-a)\times\mathbb{H}^{n-k}(-\sqrt{1-a^2})$. Thus, under the local fields of orthonormal basis $\{e_1,\cdots,e_n\}$,
$$\big(h_{ij}\big)=diag(\frac{a}{\sqrt{1-a^2}},\cdots,\frac{a}{\sqrt{1-a^2}},\frac{-\sqrt{1-a^2}}{a},\cdots,\frac{-\sqrt{1-a^2}}{a}).$$
Under the local fields of orthonormal basis, from (\ref{coff02}),  we have
$$(B_{ij})=diag(\underbrace{b_1,\cdots,b_1}_k,\underbrace{b_2,\cdots,b_2}_{n-k}),~~(A_{ij})=diag(\underbrace{a_1,\cdots,a_1}_k,\underbrace{a_2,\cdots,a_2}_{n-k}),$$
where
\begin{equation*}
\begin{split}
&b_1=\frac{1}{n}\sqrt{\frac{(n-1)(n-k)}{k}}, ~~b_2=\frac{-1}{n}\sqrt{\frac{(n-1)k}{n-k}},\\
&a_1=\frac{n-1}{k(n-k)}\frac{(n-k)^2-n^2a^2}{2n^2}, ~~~a_2=\frac{n-1}{k(n-k)}\frac{n^2a^2-n^2+k^2}{2n^2}.
\end{split}
\end{equation*}
Thus $(D^{\lambda}_{ij})=diag(\underbrace{d_1,\cdots,d_1}_k,\underbrace{d_2,\cdots,d_2}_{n-k})$ and $d_1=a_1+\lambda b_1, ~d_2=a_2+\lambda b_2$.
\begin{ex}\label{ex4}
Let $p,q$ be any two given natural numbers with $p+q<n$ and
a real number $a>1$. We define the spacelike hypersurface
$$
 x:\mathbb H^{q}(-\sqrt{a^2-1})\times\mathbb S^{p}(a)\times
 \mathbb R^+\times\mathbb R^{n-p-q-1}\rightarrow
 \mathbb R^{n+1}_1 ,$$ defined by $$
  x(u',u'',t,u''')=(tu',tu'',u'''),$$ where $u'\in \mathbb H^{q}(-\sqrt{a^2-1}), u''\in\mathbb S^{p}(a),u'''\in\mathbb R^{n-p-q-1}.$
 \end{ex}
Let $b=\sqrt{a^2-1}$. One of the normal vector of $x$ can be taken as
$$e_{n+1}=(\frac{a}{b}u',\frac{b}{a}u'',0).$$
The first and second fundamental form of $x$ are given by
$$I=t^2(<du',du'>_1+du''\cdot du'')+dt\cdot dt+du'''\cdot du''',$$
$$II=-<dx,de_{n+1}>_1=-t(\frac{a}{b}<du',du'>_1+\frac{b}{a}du''\cdot du'').$$
Thus the mean curvature of $x$ satisfies  $$H=\frac{-pb^2-qa^2}{nabt},$$
and $e^{2\tau}=\frac{n}{n-1}[\sum_{ij}h^2_{ij}-nH^2]=\frac{p(n-p)b^4-2pqa^2b^2+q(n-q)a^4}{(n-1)t^2}:=\frac{\alpha^2}{t^2}.$

From (\ref{coff0}) and (\ref{coff1}), we see that the conformal 1-form $C=0$, and the conformal metric and the conformal second fundamental form of $x$ are given by
\begin{equation}\label{metric}
\begin{split}
&g=\alpha^2<du',du'>+\alpha^2du''\cdot du''+\frac{\alpha^2}{t^2}(dt\cdot dt+du'''\cdot du''')=\tilde{g}_1+ \tilde{g}_2+ \tilde{g}_3,\\
&B=\sum_{ij}B_{ij}\omega_i\otimes\omega_j,~~~(B_{ij})=(\underbrace{b_1,\cdots,b_1}_{q},\underbrace{b_2,\cdots,b_2}_{p},\underbrace{b_3,\cdots,b_3}_{n-p-q}),\\
&A=\sum_{ij}A_{ij}\omega_i\otimes\omega_j,~~~(B_{ij})=(\underbrace{a_1,\cdots,a_1}_{q},\underbrace{a_2,\cdots,a_2}_{p},\underbrace{a_3,\cdots,a_3}_{n-p-q}),
\end{split}
\end{equation}
where $b_1=\frac{pb^2-(n-q)a^2}{nab\alpha},~b_2=\frac{qa^2-(n-p)b^2}{nab\alpha},~b_3=\frac{pb^2+qa^2}{nab\alpha},$ and
$$a_1=\frac{(pb^2+qa^2)^2-(pb^2+qa^2)2na^2+n^2a^2b^2}{2n^2a^2b^2\alpha^2},$$
$$a_2=\frac{(pb^2+qa^2)^2-(pb^2+qa^2)2nb^2+n^2a^2b^2}{2n^2a^2b^2\alpha^2},a_3=\frac{(pb^2+qa^2)^2+n^2a^2b^2}{2n^2a^2b^2\alpha^2}.$$
Thus $(D^{\lambda}_{ij})=diag(\underbrace{d_1,\cdots,d_1}_q,\underbrace{d_2,\cdots,d_2}_p,\underbrace{d_3,\cdots,d_3}_{n-p-q})$, $d_i=a_i+\lambda b_i,~~i=1,2,3.$

\begin{ex}\label{ex5}
Given constants $\lambda, r (r> 0)$, we define the spacelike hypersurface
$$x=(\frac{y_1}{y_0},\frac{y_2}{y_0}): M^k\times \mathbb{H}^{n-k}(-r) \to \mathbb{S}^{n+1}_1(1),~~2\leq k\leq n-1.$$
Here $y=(y_0,y_2):\mathbb{H}^{n-k}(-r)\to \mathbb{R}^{n-k+1}_1$ is a standard embedding, and
$y_1:M^k\to \mathbb{S}^{k+1}_1(r)\subset \mathbb{R}^{k+2}_1$
is a umbilic-free spacelike hypersurface with constant scalar curvature $R_1$ and
the mean curvature $H_1$ satisfying
$R_1=\frac{nk(k-1)+(n-1)r^2}{nr^2}-n(n-1)\lambda^2,~~~H_1=\frac{n}{k}\lambda,$ respectively.
\end{ex}
Using the structure of the spacelike hypersurface $y_1$, we have
$$\triangle_{1}y_1=-\frac{ny_1}{r^{2}}+n\lambda \xi_1,$$
where $\triangle_{1}$ is the Laplacian with respect to the first fundamental form $<dy_1,dy_1>_1$ and $\xi_1$ is the unit normal vector field of $y_1$.

The standard embedding $y=(y_0,y_2):\mathbb{H}^{n-k}(-r)\to \mathbb{R}^{n-k+1}_1$ is  totally umbilical, thus
the scalar curvature $R_{2}=-\frac{(n-k)(n-k-1)}{r^{2}},$ and
$$\triangle_{2}y=\frac{(n-k)y}{r^{2}},$$
where $\triangle_{2}$  is the Laplacian with respect to the first fundamental form $<dy,dy>_1$.

The conformal position
vector of the spacelike  hypersurface is $$Y=(y_0,y_1,y_2):M^k\times \mathbb{H}^{n-k}(-r)\rightarrow  \mathbb{R}_{2}^{n+3}.$$

Since the conformal metric $g=<dY,dY>$,
$$N=\big((-\frac{1}{2r^{2}}+\frac{\lambda^{2}}{2})y_0,(\frac{1}{2r^{2}}+\frac{\lambda^{2}}{2})y_1-\lambda \xi_1, (-\frac{1}{2r^{2}}+\frac{\lambda^{2}}{2})y_2\big).$$
We take a local orthonormal basis $\{e_{p}, p=1,...,k\}$ on  $TM^{k}$, and $\{e_{q}, q=k+1,...,n\}$ on $T\mathbb{H}^{n-k}(-r)$.  Thus
$\{e_1,\cdots,e_k,e_{k+1},\cdots,e_n\}$ is a local orthonormal basis on $T(M^k\times \mathbb{H}^{n-k}(-r))$ and
 $$Y_{i}=(0,e_{i}(y_1), \overrightarrow{0}),~1\leq i\leq k,~~ Y_{j}=(e_j(y_0),0, e_{j}(y_2)),~k+1\leq j\leq n,~~ \xi=(0,\xi_1,\overrightarrow{0}).$$
Under the local basis $\{e_1,\cdots,e_n\}$, using
$A_{ij}=<Y_{i}, N_{j}>_2,   B_{ij}=<Y_{i}, \xi_{j}>_2,$ and $D_{ij}=A_{ij}+\lambda B_{ij}$, we have $C=0$ and
\begin{equation}\label{ex5i}
\begin{split}
&(A_{ij})=(\frac{1+\lambda^{2}r^{2}}{2r^{2}}\delta_{ij}-\lambda h_{ij})\oplus(\frac{\lambda^{2}r^{2}-1}{2r^{2}}\delta_{st}),~~1\leq i,j\leq k,~~ k+1\leq s,t\leq n,\\
&(B_{ij})=(h_{ij}-\lambda\delta_{ij})\oplus(-\lambda \delta_{st}),~~~~~~~~~1\leq i,j\leq k,~~ k+1\leq s,t\leq n,\\
&(D^{\lambda}_{ij})=\frac{1-\lambda^{2}r^{2}}{2r^{2}}I_{m}\oplus (-\frac{1+\lambda^{2}r^{2}}{2r^{2}})\delta_{st},~~~1\leq i,j\leq k,~~ k+1\leq s,t\leq n.
\end{split}
\end{equation}

\begin{ex}\label{ex6}
Given constants $\lambda, r (r> 0)$, let
$$y=(y_0,\tilde{y}_0,y_1): M^k\to \mathbb{H}^{k+1}_1(-r)\subset \mathbb{R}^{k+2}_2,~~2\leq k\leq n-1,$$
be a spacelike hypersurface with constant scalar curvature $R_1$ and the mean curvature $H_1$ satisfying
$R_1=\frac{-nk(k-1)+(n-1)r^2}{nr^2}-n(n-1)\lambda^2,~H_1=\frac{n}{k}\lambda$, respectively.

 Since $-y_0^2-\tilde{y}_0^2+<y_1,y_1>=-r^2$,
$y_0$ and $\tilde{y}_0$ can not be zero simultaneously.
Without loss of generality, we assume that $y_0\neq 0$. In this case, we define the spacelike hypersurface
$$x=\big(\frac{\tilde{y}_0}{|y_0|},\frac{y_1}{|y_0|},\frac{y_2}{|y_0|}\big): M^k\times \mathbb{S}^{n-k}(r)\to \mathbb{S}^{n+1}_1(1),$$
where $y_2:\mathbb{S}^{n-k}(r)\to \mathbb{R}^{n-k+1} $ is a round sphere with radius $r$.
\end{ex}
Using the structure equation of the spacelike hypersurface $y$, we have
$$\triangle_{1}y=\frac{ky}{r^{2}}+n\lambda\xi_1,$$
where $\triangle_{1}$ is the Laplacian with respect to the first fundamental form $<dy,dy>_1$ and $\xi_1$ is the unit normal vector field of $y$.

The round sphere $y_2:\mathbb{S}^{n-k}(r)\rightarrow R^{n-k+1}$ is totally umbilical, thus
the scalar curvature is $R_{2}=\frac{(n-k)(n-k-1)}{r^{2}},$ and
$$\triangle_{2}y_2=-\frac{(n-k)y_2}{r^{2}},$$
where $\triangle_{2}$  is the Laplacian with respect to the first fundamental form $<dy_2,dy_2>$.

The conformal position
vector of the spacelike  hypersurface is
$$Y=(y,y_2):M^k\times \mathbb{S}^{n-k}(r)\to \mathbb{R}^{n+3}_2.$$
Since the conformal metric $g=<dY,dY>$, we have
$$N=((-\frac{1}{2r^{2}}+\frac{\lambda^{2}}{2})y-\lambda \xi_1, (\frac{1}{2r^{2}}+\frac{\lambda^{2}}{2})y_2).$$
We take a local orthonormal basis $\{e_{p}, p=1,...,k\}$ on $TM^k$, and $\{e_{q}, q=k+1,...,n\}$ on $T\mathbb{S}^{n-k}(r)$.  Thus
$\{e_1,\cdots,e_k,e_{k+1},\cdots,e_n\}$ is a local orthonormal basis on $T(M^k\times \mathbb{S}^{n-k}(r))$ and
 $$Y_{i}=(e_{i}(y), \overrightarrow{0}),~1\leq i\leq k,~~ Y_{j}=(\overrightarrow{0}, e_{j}(y_2)),~k+1\leq j\leq n,~~ \xi=(\xi_1,\overrightarrow{0}).$$
Under the local basis $\{e_1,\cdots,e_n\}$, using
$A_{ij}=<Y_{i}, N_{j}>_2, B_{ij}=<Y_{i}, \xi_{j}>_2,$
we have $C=0$ and
\begin{equation}\label{ex6i}
\begin{split}
&(A^{ij})=(\frac{\lambda^{2}r^{2}-1}{2r^{2}}\delta_{ij}-\lambda h_{ij})\oplus(\frac{\lambda^{2}r^{2}+1}{2r^{2}}\delta_{st}),~~1\leq i,j\leq k,~~ k+1\leq s,t\leq n,\\
&(B_{ij})=(h_{ij}-\lambda\delta_{ij})\oplus(-\lambda \delta_{st}), ~~~1\leq i,j\leq k,~~ k+1\leq s,t\leq n,\\
&(D^{\lambda}_{ij})=-\frac{1+\lambda^{2}r^{2}}{2r^{2}}\delta_{ij}\oplus \frac{1-\lambda^{2}r^{2}}{2r^{2}}\delta_{st},~~1\leq i,j\leq k,~~ k+1\leq s,t\leq n.
\end{split}
\end{equation}

In \cite{lin1}, authors classified completely the conformal isoparametric spacelike hypersurfaces in $M^{n+1}_1(c)$.
\begin{theorem}\cite{lin1}\label{th1}
Let $x: M^n\to M^{n+1}_1(c)$ be a spacelike  hypersurface in $M_1^{n+1}(c)$ with two distinct principal curvatures.  If
the conformal form vanishes, then
locally $x$
is conformally equivalent to one of the following hypersurfaces\\
(1), $\mathbb{S}^k(\sqrt{a^2+1})\times\mathbb{H}^{n-k}(-a)\subset\mathbb{S}^{n+1}_1(1), ~~a>0, ~1\leq k\leq n-1;$\\
(2), $\mathbb{H}^k(-a)\times\mathbb{H}^{n-k}(-\sqrt{1-a^2})\subset\mathbb{H}^{n+1}_1(-1),~~0<a<1,~1\leq k\leq n-1;$\\
(3), $\mathbb{H}^k(-a)\times\mathbb{R}^{n-k}
\subset\mathbb{R}^{n+1}_1,~~a>0,~1\leq k\leq n-1.$
\end{theorem}
\begin{theorem}\cite{lin1}\label{th2}
Let $x: M^n\to M^{n+1}_1(c)$ be a conformal isoparametric spacelike hypersurface in $M_1^{n+1}(c)$ with $r$ distinct principal curvatures. If $r\geq 3$, then
$r=3$, and locally $x$
is conformally equivalent to the following hypersurface
$$
 x:\mathbb H^{q}(-\sqrt{a^2-1})\times\mathbb S^{p}(a)\times
 \mathbb R^+\times\mathbb R^{n-p-q-1}\rightarrow
 \mathbb R^{n+1}_1 ,$$ defined by $$
  x(u',u'',t,u''')=(tu',tu'',u'''),$$ where $u'\in \mathbb H^{q}(-\sqrt{a^2-1}), u''\in\mathbb S^{p}(a),u'''\in\mathbb R^{n-p-q-1},~~a>1.$
\end{theorem}
The following theorem is need in the proof of the main theorem, readers  refer \cite{lin}.
\begin{theorem}\cite{lin}\label{abc1}
Let $x:M^n\rightarrow M^{n+1}_1(c)$ be a spacelike hypersurface
without umbilical points. If conformal invariants of $x$ satisfy
\begin{equation*}
(1), C=0,\;\;\; (2), A=\mu B+\lambda g,
\end{equation*}
Then $x$ is conformally equivalent to a spacelike hypersurface with
constant mean curvature and constant scalar curvature.
\end{theorem}

\par\noindent
\section{Proof of the main Theorem}
{\bf Proof of Theorem \ref{theorem1}}. From Theorem \ref{th1},  Theorem \ref{th2}, Example \ref{ex1}, Example \ref{ex2}, Example \ref{ex3} and Example \ref{ex4}, we know that
if the spacelike hypersurface is conformally isoparametric, then the spacelike hypersurface is also para-Blaschke isoparametric.

Next we assume that  the spacelike hypersurface is  a para-Blaschke isoparametric spacelike  hypersurface and the number
of the distinct eigenvalues of the para-Blaschke tensor $D^{\lambda}$ is more than two. Since the conformal $1$-form vanishes, we can have a local local orthonormal basis
$\{E_1,\cdots,E_n\}$ such that
$$(B_{ij})=diag(b_1,\cdots,b_n),~~(A_{ij})=diag(a_1,\cdots,a_n),~~(D^{\lambda}_{ij})=diag(d_1,,d_2,\cdots,d_n).$$
Using the  covariant derivative $dD^{\lambda}_{ij}+\sum_kD^{\lambda}_{kj}\omega_{ki}+\sum_kD^{\lambda}_{ik}\omega_{kj}=\sum_kD^{\lambda}_{ij,k}\omega_k$, we have
\begin{equation}\label{der1}
(d_i-d_j)\omega_{ij}=\sum_kD^{\lambda}_{ij,k}\omega_k.
\end{equation}
For each $i$ fixed, we define the index set $[i] =\{m| d_m = d_i\}.$ We have the following results
\begin{equation}\label{der2}
\begin{split}
&D^{\lambda}_{ij,k}=0,~~when~~[i]=[j], or [i]=[k], or [j]=[k].\\
&\omega_{ij}=\sum_{k\notin [i],[j]}\frac{D^{\lambda}_{ij,k}}{d_i-d_j}\omega_k,~~when~~[i]\neq [j].
\end{split}
\end{equation}
The second covariant derivative of   $D^{\lambda}_{ij}$ is defined by
 $$\sum_{l}D^{\lambda}_{ij,kl}\omega_{l}=dD^{\lambda}_{ij,k}+\sum_{l}(D^{\lambda}_{lj,k}\omega_{li}+D^{\lambda}_{il,k}\omega_{lj}+D^{\lambda}_{ij,l}\omega_{lk}).$$
 Let $[i]\neq [j]$, we have
 $$ D^{\lambda}_{ij,ij}=\sum_{k\notin [i],[j]}\frac{2(D^{\lambda}_{ij,k})^2}{d_{k}-d_{i}}, ~~
 D^{\lambda}_{ij,ji}=\sum_{k\notin [i],[j]}\frac{2(D^{\lambda}_{ij,k})^{2}}{d_{k}-d_{j}}.$$
 Using the Ricci identities
 $D^{\lambda}_{ij,ij}-D^{\lambda}_{ij,ji}=\sum_{m}(D^{\lambda}_{mj}R_{miij}+D^{\lambda}_{im}R_{mjij}),$ we get
\begin{equation}\label{der3}
R_{ijij}=\sum_{k\notin [i],[j]}\frac{2(D^{\lambda}_{ij,k})^{2}}{(d_{k}-d_{i})(d_{k}-d_{j})}.
\end{equation}
For the conformal second fundamental form $B$, we have
$$(b_{i}-b_{j})\omega_{ij}=\sum_{k}B_{ij,k}\omega_{k}.$$
Using (\ref{der2}), we get
\begin{equation}\label{der4}
(b_{i}-b_{j})\frac{D^{\lambda}_{ij,k}}{d_{i}-d_{j}}=B_{ij,k},\ \ \ [i]\neq [j].
\end{equation}
Let $k=i$ in (\ref{der4}), and using  $D^{\lambda}_{ij,i}=0$ we obtain
\begin{equation}\label{der5}
E_{j}(b_{i})=B_{ii,j}=B_{ij,i}=0, \ \  [i]\neq [j].
\end{equation}
In order to prove that $b_{i}$ is a constant, we only need to prove
$$E_{j}(b_{i})=0, \ \  [i]= [j].$$

For each $i$ fixed, we consider two cases:\\
{\bf Case 1}. There exist $j,\ k$ such that
$$ D^{\lambda}_{ij,k}\neq 0, \ \ [j]\neq [i], \ \ [k]\neq [i], \ \ [j]\neq [k].$$
{\bf Case 2}. For all $j,\ k$, we have $ D^{\lambda}_{ij,k}=0$.

Now we consider Case 1, since
$$ D^{\lambda}_{ij,k}\neq 0, \ \  [j]\neq [i], \ \ [k]\neq [i], \ \ [j]\neq [k],$$
from (\ref{der4}), we get
$$\frac{b_{i}-b_{j}}{d_{i}-d_{j}}=\frac{B_{ij,k}}{D^{\lambda}_{ij,k}}=\frac{B_{jk,i}}{D^{\lambda}_{jk,i}}
=\frac{b_{k}-b_{j}}{d_{k}-d_{j}}.$$
Thus
$$b_{i}=(b_{k}-b_{j})\frac{d_{i}-d_{j}}{d_{k}-d_{j}}+b_{j}.$$
From (\ref{der5}), since $E_{l}(b_{k})=E_{l}(b_{j})=0, \ \  l\in [i]$, we have
\begin{equation}\label{der6}
E_{l}(b_{i})=0, \ \  [i]=[l].
\end{equation}

For Case 2. Since  $ D^{\lambda}_{ij,k}=0, \ \forall j, k$, from (\ref{der3}), we get
$$R_{ijij}=0, \  \ \ j\notin [i].$$
From (\ref{stru4}) we have
$0=R_{ijij}=-b_{i}b_{j}+a_{i}+a_{j},  \ j\notin [i].$
Since $a_{i}=d_{i}-\lambda b_{i}$, we have
$$-b_{i}b_{j}+d_{i}+d_{j}=\lambda (b_{i}+b_{j}),\   j\notin [i].$$
Note that the number
of the distinct eigenvalues of the para-Blaschke tensor $D^{\lambda}$ is more than two,
we can take $j, k$ such that $[j]\neq [i],$  $[k]\neq [i]$  and $[k]\neq [j]$, and
$$-b_{i}b_{k}+d_{i}+d_{k}=\lambda (b_{i}+b_{k}),  \ j\notin [i].$$
Thus
$$-(b_{i}+\lambda)(b_{j}-b_{k})=d_{k}-d_{j}.$$
namely $$b_{i}=\frac{d_{j}-d_{k}}{b_{j}-b_{k}}+\lambda.$$
Noting $E_{l}(b_{j})=E_{l}(b_{k})=0, \ \  l\in [i], \ j, k\notin [i],$  we obtain
\begin{equation}\label{der7}
E_{l}(b_{i})=0, \ \  l\in [i].
\end{equation}
From (\ref{der5}), (\ref{der6}) and (\ref{der7}), it concludes that
$$E_{j}(b_{i})=0,\ \ 1\leq i, \ j\leq n.$$
Thus $\{b_{i}| i=1,\cdots,n\}$ are constant and $x$ is conformal isoparametric spacelike hypersurface. Thus we complete the proof of Theorem \ref{theorem1}.

Next we divide Theorem \ref{theorem2}  into three cases. If the number of the distinct eigenvalues of the para-Blaschke tensor is $1$, using Theorem \ref{abc1}, then we have the following proposition.
\begin{PROPOSITION}\label{pro1}
Let  $x:M^{n}\rightarrow M^{n+1}_1(c)$ be a para-Blaschke isoparametric spacelike hypersurface with $r$  distinct eigenvalues of the para-Blaschke tensor. If
$r=1$, then $x$ is conformally equivalent to a spacelike hypersurface with
constant mean curvature and constant scalar curvature in $M^{n+1}_1(c)$.
\end{PROPOSITION}
If the number of the distinct eigenvalues of the para-Blaschke tensor is more than two, then we have the following Proposition by Theorem \ref{th2} and Theorem \ref{theorem1}.
\begin{PROPOSITION}\label{pro2}
Let  $x:M^{n}\rightarrow M^{n+1}_1(c)$ be a para-Blaschke isoparametric spacelike hypersurface with $r$  distinct eigenvalues of the para-Blaschke tensor.
If  $r\geq 3$, then
$r=3$, and locally $x$
is conformally equivalent to the following hypersurface,
$$
 x:\mathbb H^{q}(-\sqrt{a^2-1})\times\mathbb S^{p}(a)\times
 \mathbb R^+\times\mathbb R^{n-p-q-1}\rightarrow
 \mathbb R^{n+1}_1 ,$$ defined by $
  x(u',u'',t,u''')=(tu',tu'',u'''),$ where $u'\in \mathbb H^{q}(-\sqrt{a^2-1}), u''\in\mathbb S^{p}(a),u'''\in\mathbb R^{n-p-q-1},~~a>1.$
\end{PROPOSITION}
Next we assume that the number of the distinct eigenvalues of the para-Blaschke tensor is two, we have
\begin{PROPOSITION}\label{pro3}
Let  $x:M^{n}\rightarrow M^{n+1}_1(c)$ be a para-Blaschke isoparametric spacelike hypersurface with two  distinct eigenvalues of the para-Blaschke tensor.
Then $x$ is locally conformal equivalent to one of the following hypersurfaces:\\
(1), $\mathbb{S}^k(\sqrt{a^2+1})\times\mathbb{H}^{n-k}(-a)\subset\mathbb{S}^{n+1}_1(1), ~~a>0, ~1\leq k\leq n-1;$\\
(2), $\mathbb{H}^k(-a)\times\mathbb{H}^{n-k}(-\sqrt{1-a^2})\subset\mathbb{H}^{n+1}_1(-1),~~0<a<1,~1\leq k\leq n-1;$\\
(3), $\mathbb{H}^k(-a)\times\mathbb{R}^{n-k}
\subset\mathbb{R}^{n+1}_1,~~a>0,~0\leq k\leq n-1;$\\
(4) the spacelike hypersurfaces defined by Example \ref{ex5};\\
(5) the spacelike hypersurfaces defined by Example \ref{ex6}.
\end{PROPOSITION}
\begin{proof}
Since the conformal $1$-form vanishes, we can get a local local orthonormal basis
$\{E_1,\cdots,E_n\}$ such that
$$(B_{ij})=diag(b_1,\cdots,b_n),~~~(A_{ij})=diag(a_1,\cdots,a_n),~~(D^{\lambda}_{ij})=diag(d_1,,d_2,\cdots,d_n,).$$
Furthermore, we assume that
$$d_{1}=d_{2}=...=d_{m_{1}}=\mu,~~~d_{m_{1}+1}=d_{m_{1}+2}=...= d_{n}=\nu,~~\mu\neq\nu.$$
Making the following convention on the ranges of indices:
 $$1\leq p, q, r,...\leq m_{1},\ \ m_{1}+1\leq \alpha, \beta, \gamma,...\leq n, \ \ 1\leq i, j,  k \leq n.$$
From (\ref{der1}),  for all $i, j, p, q, \alpha, \beta$ we have
\begin{equation}\label{derd}
D^{\lambda}_{ij,i}=D^{\lambda}_{ii,j}=0, \  \ D^{\lambda}_{pq,i}=D_{\alpha \beta,i}=0, \  \omega_{p\alpha}=\frac{\sum D^{\lambda}_{p\alpha,k}\omega_{k}}{\mu_{1}- \mu_{2}}.
\end{equation}
Since $\mu$ and $\nu$ are constant,  using the  total symmetry of $D_{ij,k}$, we can get that $D^{\lambda}$ is parallel, i.e., $D_{ij,k}=0, \forall i, j,k.$

Let $V_{1}$ and $V_{2}$ be the eigen-subbundles of the tangent bundle $TM$ corresponding to $\mu, \ \nu$, respectively. Then $$TM^n=V_{1}\oplus V_{2}.$$
Since $D^{\lambda}$ is parallel, we have
\begin{equation}\label{derd1}
\omega_{p\alpha}=0,~~1\leq p\leq m_1,~~m_1+1\leq\alpha\leq n,
\end{equation}
which  implies that the Riemannian manifold $(M^n,g)$ can be decomposed locally into a direct product of two Riemannian manifolds $(M_{1},g_{1})$ and
$(M_{2},g_{2})$, that is $$(M,g)=(M_{1},g_{1})\times (M_{2},g_{2}).$$
Thus $R_{p\alpha p\alpha}=0,$  and from (\ref{stru4}) we know that
\begin{equation}\label{derd2}
-b_{p} b_{\alpha}+a_{p}+a_{\alpha}=0,~~~1\leq p\leq m_1,~~m_1+1\leq\alpha\leq n.
\end{equation}

{\bf Claim 1}: {\it  The  eigenvalues of the conformal second fundamental form $B$ satisfy either
 $b_{1}=b_{2}=...=b_{m_1}$, or $b_{m_1+1}=...=b_{n}$.}

{\bf Proof of Claim 1}:  We assume that $n\geq 3$. From (\ref{derd2}), we have
 $$-b_{p} b_{\alpha}-\lambda b_{p}-\lambda b_{\alpha}+\mu+\nu=0, \ \forall p, \alpha,$$
that is
\begin{equation}\label{derd3}
-(b_{p}+\lambda)( b_{\alpha}+\lambda)+\lambda^{2} +\mu+\nu=0, \ \forall p, \alpha.
\end{equation}

If $\lambda^{2} +\mu+\nu\neq0$,
(\ref{derd3}) implies $b_{p}+\lambda\neq0, \ b_{\alpha}+\lambda\neq0, \ \forall p, \alpha.$
Since $m_1\geq 2$, we have $-(b_{q}+\lambda)( b_{\alpha}+\lambda)+\lambda^{2} +\mu+\nu=0,$
which implies that
$$(b_{p}-b_{q})( b_{\alpha}+\lambda)=0. \ \forall  p\neq q,\  \alpha .$$
Thus
$$b_{1}=b_{2}=...=b_{m_1}.$$
Similarly, if $n-m_1\geq 2$, we can obtain  $b_{m_1+1}=b_{m_1+2}=...=b_{n}$.
Thus the conformal second fundamental form  $(B_{ij})$ has two distinct constant eigenvalues.

If $\lambda^{2} +\mu+\nu=0$,
(\ref{derd3}) implies $(b_{p}+\lambda)(b_{\alpha}+\lambda)=0$, which proves the Claim 1.

By Claim 1, the proof of the Proposition \ref{pro3} is divided into the following two cases.

{\bf Case I.} The conformal second fundamental form $B$ has only  two distinct eigenvalues. According to Theorem \ref{th1}, $x$ is locally conformal equivalent to one of the hypersurfaces in Example \ref{ex1}, Example \ref{ex2} and Example \ref{ex3}.

{\bf Case II.}  The conformal second fundamental form $B$ has  more than two distinct eigenvalues.

By Claim 1, we see that either $b_{1}=b_{2}=...=b_{m_1}$, or $b_{m_1+1}=b_{m_1+2}=...=b_{n}$.
Without loss of generality, assuming $b_{m_1+1}=b_{m_1+2}=...=b_{n}$. From the proof of the Claim 1, we know that
$b_{m_1+1}=b_{m_1+2}=...=b_{n}=-\lambda$,
$m_1\geq 2$ and $\lambda^{2} +\mu+\nu=0$.

Let $\tilde{h}_{pq}=B_{pq}+\lambda\delta_{pq}$, then from (\ref{stru4}), we obtain the components of curvature tensor on  $(M_{1},g_{1})$
\begin{equation}\label{derd4}
 \begin{array}{lll}
 R_{pqst}&=&-B_{ps}B_{qt}+B_{pt}B_{qs}+(\mu\delta_{ps}-\lambda B_{ps})\delta_{qt}+(\mu\delta_{qt}-\lambda B_{qt})\delta_{ps}\\
 &-&(\mu\delta_{pt}-\lambda B_{pt})\delta_{qs}-(\mu\delta_{qs}-\lambda B_{qs})\delta_{pt}\\
 &=&(2\mu+\lambda^{2})(\delta_{ps}\delta_{qt}-\delta_{pt}\delta_{qs})+\tilde{h}_{pt}\tilde{h}_{qs}-\tilde{h}_{ps}\tilde{h}_{qt}.
 \end{array}
 \end{equation}
 The components of curvature tensor on  $(M_{2},g_{2})$
\begin{equation}\label{derd5}
R_{\alpha \beta\gamma\eta}=(2\nu+\lambda^{2})(\delta_{\alpha \gamma}\delta_{\beta \eta}-\delta_{\beta \gamma}\delta_{\alpha \eta}).
\end{equation}
Hence if $n-m_1\geq 2$,  $(M_{2},g_{2})$  is of constant sectional curvature $2\nu+\lambda^{2}$.

Since $(2\mu+\lambda^{2})+(2\nu+\lambda^{2})=2(\lambda^{2} +\mu+\nu)=0$ and $\mu\neq\nu$, we need to consider the following two subcases:

{\bf Subcase 2.1}. $2\mu+\lambda^{2}>0, \ 2\nu+\lambda^{2}<0.$

Set $r=(2\mu+\lambda^{2})^{-\frac{1}{2}}$, then $ 2\nu+\lambda^{2}=-r^{-2},$ and $(M_{2},g_{2})$  can be locally identified with $\mathbb{H}^{n-m_1}(-r)$.
Let $v:\mathbb{H}^{n-m_1}(-r)\rightarrow \mathbb{R}_{1}^{n-m_1+1}$ be the standard totally umbilical hypersurface.

Writing $h^{1}=\sum_{p,q=1}^{m_1}\tilde{h}_{pq}\omega_{p}\otimes \omega_{q}$, by $C=0$ and (\ref{stru2}), we know $h^{1}$ is a Codazzi tensor on
 $(M_{1},g_{1})$,  (\ref{derd4}) means that there exists a space-like hypersurface
 $$u:N^{m_1}\rightarrow \mathbb{S}_{1}^{m_1+1}(r)\subset \mathbb{R}_{1}^{m_1+2},$$
 with $h^{1}$ as its second fundamental form. Clearly, $u$ has at least two non-zero principal curvatures.
According to  (\ref{derd4}) and (\ref{cond1}), we can prove directly that
$u$ is of constant mean curvature $H_{1}$ and  constant scalar curvature $R_{1}$ satisfying
$$H_{1}=\frac{n\lambda}{m_1},\ \ R_{1}=\frac{m_1(m_1-1)}{r^{2}}+\frac{n-1}{n}-n(n-1)\lambda^{2}.$$
Thus $x$ is locally conformal equivalent to the hypersurfaces in Example \ref{ex5}.

{\bf Subcase 2.2.} $2\mu+\lambda^{2}<0, \ 2\nu+\lambda^{2}>0.$

Set $r=(2\nu+\lambda^{2})^{-\frac{1}{2}}$, then $ 2\mu+\lambda^{2}=-r^{-2},$ and $(M_{2},g_{2})$  can be locally identified with $\mathbb{S}^{n-m_1}(r)$.
Let $v:\mathbb{S}^{n-m_1}(r)\rightarrow \mathbb{R}^{n-m_1+1}$ be the standard totally umbilical hypersurface.

Writing $h^{1}=\sum_{pq=1}^{m_1}\tilde{h}_{pq}\omega_{p}\otimes \omega_{q}$, by $C=0$ and (\ref{stru2}), we know $h^{1}$ is a Codazzi tensor on
 $(M_{1},g_{1})$,  (\ref{derd4}) means that there exists a space-like hypersurface
 $$u:N^{m_1}\rightarrow \mathbb{H}_{1}^{m_1+1}(-r)\subset \mathbb{R}_{2}^{m_1+2},$$
 with $h^{1}$ as its second fundamental form. Clearly, $u$ has at least two non-zero principal curvatures.
According to (\ref{derd4}) and (\ref{cond1}),  we can prove directly  that
$u$ is of constant mean curvature $H_{1}$ and  constant scalar curvature $R_{1}$ satisfying
$$H_{1}=\frac{n\lambda}{m_1},\ \ R_{1}=-\frac{m_1(m_1-1)}{r^{2}}+\frac{n-1}{n}-n(n-1)\lambda^{2}.$$
Thus $x$ is locally conformal equivalent to the hypersurfaces in Example \ref{ex6}.

Thus we complete the proof of Proposition \ref{pro3}
\end{proof}
Using Proposition \ref{pro1}, Proposition \ref{pro2} and Proposition \ref{pro3}, we finish the proof of Theorem \ref{theorem2}.

{\bf Acknowledgements:} Authors are supported by the
grant No. 11571037 and No. 11471021 of NSFC.


\begin{thebibliography}{11}
\bibitem{cahne}M. Cahen, Y. Kerbrat, {\sl Domaines symm\'{e}triques des quadriques projectives}, J. Math. Pure Appl., 62(1983), 327-348.
\bibitem{cheng}Q. M. Cheng, X. X. Li, X. R. Qi, {\sl A classification of hypersurfaces with parallel para-Blaschke tensor in $S^{m+1}$,} Int. J.
Math., 21(2010), 297-316.
\bibitem{guo}Z. Guo, J. B. Fang, L. M. Lin, {\sl Hypersurfaces with isotropic Blaschke tensor,} J. Math. Soc. Japan, 4(2011), 1155-1186.
\bibitem{guo1}Z. Guo, H. Li, C. P. Wang, {\sl The M\"{o}bius characterizations of Willmore tori and
Veronese submanifolds in unit sphere,} Pacific J. Math. 241(2009), 227-242.
\bibitem{hu}Z. J. Hu, H. Z. Li, {\sl Submanifolds with constant M\"{o}bius scalar curvature in $S^n$,}
Manuscripta Math., 111 (2003), 287-302.
\bibitem{hu1}Z. J. Hu, H. Z. Li, {\sl Classification of hypersurfaces with parallel M\"{o}bius second fundamental form in $S^{n+1}$,} Sci. China
Ser. A, 47(2004), 417-430(2004).
\bibitem{lin}T. Z. Li, C. X. Nie, {\sl  Conformal Geometry of Hypersurfaces in Lorentz Space Forms},  Geometry,  Volume 2013,  Article ID
549602.
\bibitem{lin1} T. Z. Li, C. X. Nie, {\sl Spacelike Dupin Hypersurfaces in Lorentzian
Space Forms}, Preprint.(to appear in Journal of the Mathematical Society of Japan.)
\bibitem{lih}H. Z. Li, C. P. Wang, {\sl M\"{o}bius geometry of hypersurfaces with constant mean curvature and scalar
curvature,} Manuscripta Math., 112 (2003), 1-13.
\bibitem{lih3}H. Z. Li, {\sl Willmore hypersurfaces in a sphere,} Asian J. Math., 5 (2001), 365¨C377.
\bibitem{lix1}X. X. Li, H. R. Song, {\sl On the regular space-like hypersurfaces with parallel Blaschke tensors in the de Sitter space $S^{m+1}_1$},
arXiv [math. DG]: 1511.02560, 2015; to appear.
\bibitem{lix4}X. X. Li, H. R. song, {\sl Regular Spacelike hypersurfaces in $S^{m+1}_1$ with parallel para-Blaschcke tensor},arXiv [math. DG]: 1511.03261.
\bibitem{lix2}X. X. Li, F. Y. zhang, {\sl Immersed hypersurfaces in the unit sphere $S^{m+1}$ with constant Blaschke eigenvalues,} Acta
Math. Sinica, (Eng. ser.), 23(2007), 533-548.
\bibitem{lix3}X. X. Li, F. Y. Zhang, {\sl A classification of immersed hypersurfaces in spheres with parallel Blaschke tensors,} Tohoku
Math. J., 58(2006), 581-597.
\bibitem{nie}C. X. Nie, T. Z. Li, Y. J. He, et al.: {\sl Conformal isoparametric hypersurfaces with two distinct conformal
principal curvatures in conformal space.} Sci. China Ser. A, 53(2010), 953-965.
\bibitem{nie1}C. X. Nie, {\sl Blaschke Isoparametric Hypersurfaces in the Conformal Space $Q^{n+1}_1$ I},
Acta Math. Sinica, English series, 31(2015), 1751-1758.
\bibitem{nie2}C. X. Nie, C. X. Wu, {\sl Spacelike hypersurfaces with parallel conformal second fundamental forms in the conformal
space,} Acta Math. Sinica, Chinese series, 51(2008), 685-692.
\bibitem{o}O'Neil B., {\sl Semi-Riemannian Geometry}, Academic Press, New York(1983).
\bibitem{w} C. P. Wang,  {\sl M\"{o}bius geometry of submanifolds in $\mathbb{S}^n$},  Manuscripta
Math., 96(1998), 517-534.
\end{thebibliography}
\end{document}